\documentclass[12pt,reqno]{article}

\usepackage{amsthm}
\usepackage{hyperref}
\usepackage{graphics,amsmath,amssymb} 

\usepackage{color}
\usepackage{fullpage}
\usepackage{float}
\usepackage{epsf}

\begin{document}
\begin{center}

\vskip 1.5cm {\large \bf  DECOMPOSITION OF AN INTEGER AS     \\
THE SUM OF TWO CUBES TO A FIXED MODULUS}

\vskip 1cm
\large

David Tsirekidze\\
Stanford University\\
USA \\
\href{mailto: david19@stanford.edu}{\tt david19@stanford.edu}\\
\vskip 1cm
Ala Avoyan\\
Tbilisi State University\\
Republic of Georgia\\
\href{mailto: a.avoyan@iset.ge}{\tt a.avoyan@iset.ge}\\
\end{center}
\ \\

\bigskip

\vspace*{0.25cm} \noindent {\small {\bf Abstract}. {The representation of any integer as the sum of two cubes to a fixed 
modulus is always possible if and only if the modulus is not divisible 
by seven or nine. For a positive non-prime integer $N$ there is given an
inductive way to find its remainders that can be represented as the sum 
of two cubes to a fixed modulus $N$. Moreover, it is possible to find 
the components of this representation.}

\vspace*{0.25cm}\noindent{\small {\bf AMS subject classification
(2000):} { 11A07, 11B50, 11D25,}}

\vspace*{0.25cm}\noindent{\small {\bf Keywords and phrases}:
{sum of two cubes; diophantine equation.} }

\newtheorem{theorem}{Theorem}[section]
\newtheorem{defin}{Definition}[section]
\newtheorem{remark}{Remark}[section]
\vskip 1 cm


\section{Introduction}\label{S1} 
Any odd prime number, $p$, can be written as the sum of two squares if and only if it is of the form $p = 4k+1$,  where $k \in N$. Generally, number $n$ can be represented as a sum of two squares if and only if in the prime factorization of $n$, every prime of the form $4k+3$ has even exponent. There is no such a nice characterization for the sum of two cubes. 
In this paper we give an inductive way which allows to find the representation of a non-prime integer as a sum of two cubes to a given modulus.\\
\begin{defin} 
$N\ge 2$ let
$$
\delta(N)=\frac{\#\{n\in\{1,\dots,N\}:n\equiv x^3+y^3 \pmod N
\text{\space has a solution}\}}{N}.
$$
\end{defin}

Broughan \cite{broughan} proved the following theorem\\

\begin{theorem}
1. If $7\mid N$ and $9\nmid N$ then $\delta(N)=5/7$\\
2. If $7\nmid N$ and $9\mid N$ then $\delta(N)=5/9$\\
3. If $7\mid N$ and $9\mid N$ then $\delta(N)=25/63$\\
4. If $7\nmid N$ and $9\nmid N$ then $\delta(N)=1$.\\
\end{theorem}
In the last case $\delta(N)=1$ , therefore, any integer can be represented as a sum of two cubes to a fixed modulus $N$.

By the Theorem 1.1 for all $N$ we can deduce number of its remainders that can be decomposed as a sum of two cubes. In this paper we introduce the way to find these remainders and also their decompositions as a sum of two cubes to a fixed modulus $N$ in case when we know the factorization of this number.

\section{Main Results}\label{S2} 
\begin{theorem}
Let us consider an equation $n\equiv u^3+v^3 \pmod N$, $n \in [0,N-1]$. Then it has solution in Z in the following congruences:\\ 
1. $7\mid N$, $9\nmid N$ and $n \equiv 0, 1, 2, 5, 6 \pmod 7$; \\
2. $7\nmid N$, $9\mid N$ and $n \equiv 0, 1, 2, 7, 8 \pmod 9$; \\
3. $7\mid N$, $9\mid N$ and $n \equiv 0, 1, 2, 7, 8, 9, 16, 19, 20, \\
26, 27, 28, 29, 34, 35, 36, 37, 43, 44, 47, 54, 55, 56, 61, 62 \pmod {63}$; \\
4.  $7\nmid N$ and $9\nmid N$.
\end{theorem}
\begin{proof}
For simplicity, we prove only the first case of the theorem. One can easily verify that cube of any integer number can have the following remainders modulo 7: 0, 1, 6. Therefore, the sum of two cubes can have remainders 0, 1, 2, 5, 6 modulo 7. The number of positive integers with these remainders is $(5/7) \cdot N$ in the interval $[0, N-1]$. There is no other number $n$ for which the equation has a solution. Hence, from Theorem 1.1 the first case of Theorem 2.1 is proved. Other two cases can be proved analogously. 
\end{proof}

\begin{defin}
Let us denote the set of all values of $n$ for which $n\equiv u^3+v^3 \pmod N$ by $A(N)$.
\end{defin}

\begin{theorem}
If $(N,M)=1$, then $\delta\ (MN)=\delta(M)\cdot\delta (N)$
\end{theorem}

\begin{proof}
\textrm{Suppose}
\begin{equation}
m\equiv u^3+v^3 \pmod M \          \ m \in [0,M-1] \\
\end{equation}
\begin{equation}
n\equiv x^3+y^3 \pmod N \          \ n \in [0,N-1] \\
\end{equation}

Let $X$ be such that $M\mid X$ and $N\mid X-1$. By the Chinese Remainder Theorem such $X$ always exists.

Let us construct $X^{*}, A$ and $B$ by the following way
\begin{equation}
X^{*}\equiv X\cdot n - (X-1)\cdot m\pmod {MN} 
\end{equation}
\begin{equation}
A = X\cdot x -(X-1)\cdot u 
\end{equation}
\begin{equation}
B = X\cdot y -(X-1)\cdot v 
\end{equation}

It can be shown that $X^{*}\equiv A^3+B^3\pmod {MN}$.\\
Indeed, 
$$
X^{*}-(A^{3}+B^{3})\equiv
$$
$$
\equiv X\cdot n-(X-1)\cdot m-(X^{3}\cdot x^{3}-(X-1)^{3}\cdot u^{3}+X^{3}\cdot y^{3}-(X-1)^{3}\cdot v^{3})\equiv \\
$$
$$
\equiv X\cdot n-(X-1)\cdot m -(X^{3}(x^{3}+y^{3})-(X-1)^{3}(u^{3}+v^{3}))\equiv \\
$$
$$
\equiv X\cdot (n-X^{2}( x^{3}+y^{3}))+(X-1)\cdot ((X-1)^{2}(u^{3}+v^{3})-m)\pmod{MN} \\ 
$$
As, 
$$
n-X^{2}(x^{3}+y^{3})\equiv (x^3+y^3)(1-X)(1+X) \equiv 0\pmod N \ \textrm{and} \ X\equiv 0\pmod M \\
$$
And $(N,M)=1$, we obtain \\
$$
X\cdot (n-X^{2}( x^{3}+y^{3}) )\equiv 0\pmod {MN}\\
$$
Analogously,\\
$$
(X-1)^{2}(u^{3}+v^{3})-m\equiv (u^{3}+v^{3})\cdot ((X-1)^{2}-1)\equiv 0\pmod M \ \textrm{and} \\ \ X-1 \equiv 0\pmod N\\
$$
Consequently, as $(N,M)=1$ \\
$$
(X-1)\cdot((X-1)^{2}(u^{3}+v^{3})-m)\equiv 0 \pmod {MN}\\
$$
Finally, \\
$$
X^{*}-(A^{3}+B^{3})\equiv X\cdot (n-X^{2}(x^{3}+y^{3}))+(X-1)\cdot((X-1)^{2}(u^{3}+v^{3})-m) \equiv 0\pmod {MN}\\
$$
\vskip 1cm
For any $m\in A(M)$ and any $n\in A(N)$, there exists $X^{*}\in A(MN)$.
Obviously, $X^{*}\equiv n\pmod N$ and $X^{*}\equiv m\pmod M$. Thus, for different
pairs $(m_{1},n_{1})$ and $(m_{2},n_{2})$ we cannot obtain the same $X^{*}$
(by Chinese Remainder Theorem).\\
Now take any element from $A(MN)$ set,
$X^{*}\equiv A^{3}+B^{3}\pmod{MN}$. Suppose $(x,y), (u,v)$ pairs are the solutions of
the following Diophantine equation:
$$
A=X\cdot x-(X-1)\cdot u
$$
$$
B=X\cdot y-(X-1)\cdot v.
$$
If we consider 
\begin{center}
$m\equiv(u^{3}+v^{3})\pmod M \ and \ n\equiv(x^{3}+y^{3})\pmod N$.
\end{center}
Then $X^{*}\equiv A^{3}+B^{3}\pmod{MN}$. Therefore, there is one-to-one correspondence between the elements of the set $A(MN)$ and pairs of elements from the sets $A(M)$ and $A(N)$. Hence, we proved that $\delta(MN)=\delta(M)\cdot\delta(N)$
\end{proof}

\begin{remark}
Let us assume we are given any number $K$ and suppose we know the representation of any element in each set $A(1), A(2),...,A(K-1)$ as a sum of two cubes to a fixed modulus. And our task is to find the representation of the elements of $A(K)$. Let K be a non-prime number and $K=M \cdot N$, where $(M,N)=1$ and $N,M>1$. Suppose $m \in A(M)$, $n \in A(N)$ and (1),(2) hold.
Solve Diophantine equation $M\cdot q - N \cdot l=1$, let $X=Mq$ and construct $X^{*}, A, B$ according to (3),(4),(5). As it was shown above 
\begin{equation}
X^{*}\equiv A^3+B^3\pmod K
\end{equation}
Therefore $X^{*}\in A(K)$ and (6) is the representation for $X^{*}$ as a sum of two cubes to a fixed modulus $K$.
\end{remark}


\section{Conclusion}\label{S3} 

This paper is an attempt to explicitly find the way to solve equation $n \equiv a^3+b^3 \pmod K$. Using inductive way that is given in this paper it is possible to construct $A(K)$ set and represent any element of this set as a sum of two cubes to a fixed non-prime modulus $K$. For further research this issue can be considered for prime number $K$.

\end{document}